\documentclass[12pt]{article}

\usepackage{amsmath,amssymb,amsthm}
\usepackage{mathtools}
\usepackage{geometry}
\newtheorem{theorem}{Theorem}[section]

\newtheorem{proposition}[theorem]{Proposition}
\newtheorem{remark}[theorem]{Remark}
\newtheorem{definition}[theorem]{Definition}

\newtheorem{corollary}[theorem]{Corollary}

\title{Heat--Airy Transport of Differential Operators and Moving Boundaries}

\author{Gerardo Hern\'andez-del-Valle
\thanks{Centro de Estudios Monetarios Latinoamericanos (CEMLA), Mexico City, Mexico.
E-mail: \texttt{ghernandez@cemla.org}}
}

\begin{document}
\date{}
\maketitle

\begin{abstract}
We study the transport of polynomial differential operators under the
two-parameter commuting evolution
\[
P_{t,s}
=
\exp\left(
\frac{t}{2}D^2-\frac{s}{3}D^3
\right),
\qquad
D=\frac{d}{dx}.
\]
Conjugation of the position operator gives
\[
P_{t,s}xP_{t,s}^{-1}
=
x+tD-sD^2,
\]
which leads to a recursive normal-ordering expansion and to a family of
Heat--Airy polynomials arising as its derivative-free coefficients.

Our main purpose is to study the interaction of this transport with moving
absorbing boundaries. For functions satisfying
\[
v_t=\frac12v_{xx},
\qquad
v_s=-\frac13v_{xxx},
\qquad
v(t,s,f(t,s))=0,
\]
the boundary condition generates a hierarchy of relations among the
spatial jets of \(v\). Combined with the restrictions of transported
differential equations and their spatial derivatives, these identities
produce compatibility equations for \(f\).

We prove a general restriction principle showing that, at every spatial-jet
level, the compatibility equations of the classical Heat problem are
obtained by restricting the corresponding Heat--Airy hierarchy to the
slice \(s=0\). For second-order operators, this framework yields a finite
characteristic system and nonlinear boundary equations, together with
explicit solvable examples. For a third-order linear-potential operator,
the additional Airy-flow identity permits a stronger sequential recovery:
the first Heat--Airy compatibility equation determines \(f_s(t,0)\), and
the next spatial-jet equation then reproduces the nonlinear compatibility
condition obtained in the pure Heat theory from a \(4\times4\)
characteristic determinant. Thus the additional commuting flow reorganizes
the moving-boundary compatibility problem into a hierarchy while retaining
the classical Heat theory as a distinguished restriction.
\end{abstract}

\eject 
\noindent absorbing boundaries; normal ordering; Heat--Airy polynomials;
boundary-jet compatibility.

\medskip


\section{Introduction}
\label{sec:Introduction}

The transport of differential operators by the heat semigroup provides a
natural connection between polynomial differential equations, polynomial
solution families, and moving-boundary problems. Let
\[
P_t
=
\exp\left(\frac{t}{2}D^2\right),
\qquad
D=\frac{d}{dx}.
\]
Since
\[
P_tDP_t^{-1}=D,
\qquad
P_txP_t^{-1}=x+tD,
\]
conjugation of a polynomial differential operator by \(P_t\) can be
described algebraically by leaving \(D\) unchanged and replacing the
position operator \(x\) by the noncommutative operator \(x+tD\). When this
transport is combined with an absorbing boundary condition, identities
satisfied by the transported equation can be converted into differential
compatibility conditions for the moving boundary.

Several established theories meet naturally in this construction. Heat
polynomials and their associated functions were studied systematically by
Rosenbloom and Widder \cite{Widder}. Higher-order extensions of Hermite
polynomials and their operational representations go back, in particular,
to Gould and Hopper \cite{GouldHopper}; see also Dattoli et al.\
\cite{DattoliHermite}. Higher-order heat-type equations and their
connections with generalized Airy functions and higher-order polynomial
families have likewise been investigated through operational and
integral-transform methods; see G\'orska et al.\ \cite{Gorska}. From the
algebraic point of view, the noncommutativity of multiplication by \(x\)
and differentiation places the resulting expansions within the general
framework of normal ordering; see B{\l}asiak et al.\ \cite{Blasiak} and
the references therein.

A second source of motivation comes from moving-boundary problems for the
heat equation. De Lillo and Fokas \cite{fokas} studied the
Dirichlet-to-Neumann map for the heat equation on a moving boundary. In a
probabilistic setting, such problems arise naturally in the study of
first-passage times for Brownian motion. In \cite{hernandez}, Brownian
hitting-time problems were related to Bessel bridges, heat equations, and
Schr\"odinger-type differential equations. This line of work was developed
further in Hern\'andez-del-Valle and Guerra-Polania \cite{HerGuerra},
where heat equations with moving boundaries were studied directly in
connection with Brownian hitting times.

Higher-order spectral structures have also appeared in this context. In
\cite{pearcey}, zeros of the Pearcey integral were analyzed through heat
evolution together with differential identities inherited from the
underlying higher-order integral representation, leading to a nonlinear
Rayleigh-type equation for the corresponding zero curves. More recently,
the classical heat transport of polynomial differential operators was
organized through compatibility matrices and their recursive structure
\cite{HernandezPreprint}. The present paper continues this program by
promoting the higher-order evolution itself to an independent commuting
flow and studying how the additional flow reorganizes the moving-boundary
compatibility hierarchy.

More precisely, we consider
\[
P_{t,s}
=
\exp\left(
\frac{t}{2}D^2-\frac{s}{3}D^3
\right).
\]
Since \(D^2\) and \(D^3\) commute,
\[
P_{t,s}
=
\exp\left(\frac{t}{2}D^2\right)
\exp\left(-\frac{s}{3}D^3\right),
\]
and the parameters \(t\) and \(s\) describe two commuting evolutions. The
corresponding conjugation of the position operator is
\[
P_{t,s}xP_{t,s}^{-1}
=
X_{t,s}
=
x+tD-sD^2.
\]
Thus the classical transported position operator \(x+tD\) is replaced by
the second-order operator \(x+tD-sD^2\).

The action of \(P_{t,s}\) on monomials produces the polynomial family
\[
\mathcal H_n(x,t,s)
=
P_{t,s}(x^n),
\]
with exponential generating function
\[
\sum_{n=0}^{\infty}
\frac{\mathcal H_n(x,t,s)}{n!}z^n
=
\exp\left(
xz+\frac{t}{2}z^2-\frac{s}{3}z^3
\right).
\]
These polynomials lie naturally within the operational theory of
generalized Hermite and higher-order heat polynomials
\cite{GouldHopper,DattoliHermite,Gorska}. They satisfy
\[
\partial_t\mathcal H_n
=
\frac12D^2\mathcal H_n,
\qquad
\partial_s\mathcal H_n
=
-\frac13D^3\mathcal H_n,
\]
and reduce to the classical heat polynomials when \(s=0\).

For the transport of differential operators, the derivative-free
polynomials \(\mathcal H_n\) constitute only part of the structure. We
therefore consider the complete normal ordering
\[
X_{t,s}^n
=
\sum_{k=0}^{2n}
T_{n,k}(x,t,s)D^k
\]
and derive a recursion for the coefficients \(T_{n,k}\). They form a
triangular array of noncommutative correction terms, whose first column is
\[
T_{n,0}(x,t,s)=\mathcal H_n(x,t,s).
\]
This normal-ordering recursion provides a systematic mechanism for
transporting polynomial differential operators without expanding powers
of \(X_{t,s}\) individually.

The principal object of the paper is the moving-boundary compatibility
structure associated with this transport. Let \(v=v(t,s,x)\) satisfy
\[
v_t=\frac12v_{xx},
\qquad
v_s=-\frac13v_{xxx},
\]
and suppose that
\[
v(t,s,f(t,s))=0.
\]
Differentiation of the absorbing condition along the two commuting flows
gives
\[
v^{(2)}=-2f_tv^{(1)},
\qquad
v^{(3)}=3f_sv^{(1)},
\]
and further differentiation produces relations among successively higher
spatial jets of \(v\). If, in addition,
\[
N_{t,s}v=0,
\qquad
N_{t,s}=P_{t,s}LP_{t,s}^{-1},
\]
then the restrictions
\[
\gamma_f\!\left(D^r(N_{t,s}v)\right)=0,
\qquad
r\geq0,
\]
supply a second hierarchy of relations among the same boundary jets.
Compatibility between these two families imposes differential constraints
on \(f\).

A basic structural result of the paper concerns the relation between this
two-parameter hierarchy and the classical heat problem. Let
\[
F(t)=f(t,0),
\qquad
u(t,x)=v(t,0,x),
\]
and let
\[
N_t^H
=
P_tLP_t^{-1}.
\]
Since
\[
P_{t,0}=P_t,
\qquad
N_{t,0}=N_t^H,
\]
we prove that, for every integer \(r\geq0\),
\[
\left.
\gamma_f\!\left(
D^r(N_{t,s}v)
\right)
\right|_{s=0}
=
\gamma_F\!\left(
D^r(N_t^Hu)
\right).
\]
Thus the spatial-jet hierarchy generated by the transported differential
equation in the classical heat problem is embedded in the corresponding
Heat--Airy hierarchy through the distinguished slice \(s=0\). Together
with the restriction of the moving-boundary identities, this recovers the
classical heat compatibility structure. The restriction principle is
independent of the order of the polynomial differential operator \(L\).

The enlarged hierarchy nevertheless carries additional information.
Indeed, relations involving \(f_s\), such as
\[
v^{(3)}=3f_sv^{(1)},
\]
have no counterpart in the one-parameter heat formulation. Consequently,
information that in the pure heat theory is extracted through simultaneous
elimination of several spatial jets may, for suitable operators, be
resolved successively by using the additional Airy-flow identities. The
general restriction theorem does not assert that such a simplification
always occurs, but it provides the framework in which this possibility can
be analyzed.

For second-order operators, the Heat--Airy transport contains derivatives
of order at most four. The transported equation and the boundary identities
therefore lead to a finite compatibility system for
\[
\left(
v^{(4)},v^{(3)},v^{(2)},v^{(1)}
\right)^{\mathsf T}.
\]
The vanishing of the corresponding compatibility determinant yields a
nonlinear partial differential equation for \(f(t,s)\), whose restriction
to \(s=0\) recovers the corresponding heat compatibility equation.
Explicit specializations show that the first compatibility relation may
determine the \(t\)-dependence of \(f\) only up to an arbitrary function of
\(s\), while the next spatial-jet equation removes the remaining freedom.
A second example shows how regularity at the heat slice and higher
spatial-jet compatibility determine the first-order deformation of a
classical heat boundary in the Airy direction.

The third-order case exhibits a stronger form of this phenomenon. For the
sparse operator
\[
L_3
=
c_{3,0}D^3+c_{0,0}+c_{0,1}x,
\]
the first Heat--Airy compatibility equation is
\[
3c_{3,0}f_s
+
2c_{0,1}s f_t
+
c_{0,1}t
=
0.
\]
This equation can be solved explicitly by characteristics and propagates
the boundary datum \(F(t)=f(t,0)\) into the additional \(s\)-direction. In
particular,
\[
f_s(t,0)
=
-\frac{c_{0,1}}{3c_{3,0}}t.
\]
The zeroth-order coefficient \(c_{0,0}\), which disappears from the first
equation because \(v=0\) on the boundary, reappears at the next spatial-jet
level when
\[
D(N_3v)=0
\]
is restricted to the boundary. Substitution of the preceding expression
for \(f_s(t,0)\) into this second relation yields exactly the nonlinear
compatibility equation obtained independently in the pure heat theory from
the corresponding \(4\times4\) compatibility determinant.

This sequential recovery is the main illustration of the role played by
the additional commuting flow. The Heat--Airy extension does not simply
add another parameter to the classical problem. The additional commuting
flow enlarges the available boundary identities and can reorganize a
simultaneous spatial-jet elimination into successive levels of a
compatibility hierarchy. At the same time, the general restriction
principle ensures that the classical heat compatibility structure remains
embedded in the enlarged theory.

The paper is organized as follows.
Section~\ref{sec:HeatAirySemigroup} introduces the Heat--Airy semigroup and
its conjugation properties.
Section~\ref{sec:HeatAiryPolynomials} studies the associated polynomial
family.
Section~\ref{sec:NormalOrdering} derives the recursive normal-ordering
expansion and applies it to polynomial differential operators.
Section~\ref{sec:MovingBoundary} develops the moving-boundary spatial-jet
relations and the second-order compatibility system.
Section~\ref{sec:ExplicitBoundaries} analyzes explicit solvable examples.
Section~\ref{sec:ThirdOrder} establishes the general restriction principle
and studies the third-order sequential-recovery phenomenon.
Finally, Section~\ref{sec:Conclusion} discusses higher commuting flows and
possible extensions.

\section{The Heat--Airy semigroup}
\label{sec:HeatAirySemigroup}

The classical heat semigroup provides a natural mechanism for transporting
functions and differential operators. In the present setting, we enlarge
this construction by adjoining a commuting evolution generated by the
third spatial derivative. The resulting two-parameter family will provide
the analytic and algebraic framework used throughout the paper.

Let
\[
D=\frac{d}{dx}.
\]

\begin{definition}
The \emph{Heat--Airy semigroup} is the two-parameter family of operators
\[
P_{t,s}
=
\exp\left(
\frac{t}{2}D^2-\frac{s}{3}D^3
\right),
\qquad
t,s\geq0.
\]
\end{definition}
Since \(D^2\) and \(D^3\) commute, the two flows factorize:
\[
P_{t,s}
=
\exp\left(\frac{t}{2}D^2\right)
\exp\left(-\frac{s}{3}D^3\right)
=
\exp\left(-\frac{s}{3}D^3\right)
\exp\left(\frac{t}{2}D^2\right).
\]

This immediately gives the two-parameter semigroup property.

\begin{proposition}
For \(t_1,t_2,s_1,s_2\geq0\),
\[
P_{t_1,s_1}P_{t_2,s_2}
=
P_{t_1+t_2,s_1+s_2}.
\]
Moreover,
\[
\partial_tP_{t,s}
=
\frac12D^2P_{t,s},
\qquad
\partial_sP_{t,s}
=
-\frac13D^3P_{t,s}.
\]
\end{proposition}

\begin{proof}
Since
\[
[D^2,D^3]=0,
\]
we have
\[
\begin{aligned}
P_{t_1,s_1}P_{t_2,s_2}
&=
\exp\left(
\frac{t_1}{2}D^2-\frac{s_1}{3}D^3
\right)
\exp\left(
\frac{t_2}{2}D^2-\frac{s_2}{3}D^3
\right)
\\
&=
\exp\left(
\frac{t_1+t_2}{2}D^2
-
\frac{s_1+s_2}{3}D^3
\right)
\\
&=
P_{t_1+t_2,s_1+s_2}.
\end{aligned}
\]
The differential identities follow directly by differentiation with
respect to the corresponding parameters.
\end{proof}

The Fourier representation makes the interaction between the two flows
particularly transparent. With the convention
\[
\widehat{f}(\lambda)
=
\int_{\mathbb R}e^{-i\lambda x}f(x)\,dx,
\]
the operator \(D\) is represented by multiplication by \(i\lambda\).
Consequently,
\[
\widehat{P_{t,s}f}(\lambda)
=
\exp\left(
-\frac12\lambda^2t
+
i\frac{\lambda^3}{3}s
\right)
\widehat f(\lambda).
\]

This leads to the following kernel.

\begin{definition}
The \emph{Heat--Airy kernel} is
\[
K(t,s,x)
=
\frac{1}{2\pi}
\int_{\mathbb R}
\exp\left(
i\lambda x
-\frac12\lambda^2t
+i\frac{\lambda^3}{3}s
\right)
\,d\lambda.
\]
\end{definition}

For \(t>0\), the Gaussian factor provides decay in the Fourier variable,
while the cubic term contributes an oscillatory phase. The kernel satisfies
the pair of commuting evolution equations
\[
K_t=\frac12K_{xx},
\qquad
K_s=-\frac13K_{xxx}.
\]

Accordingly, for suitable functions \(f\),
\[
(P_{t,s}f)(x)
=
\int_{\mathbb R}
K(t,s,x-y)f(y)\,dy.
\]

The central algebraic object in what follows is obtained by conjugating
the multiplication operator \(x\) by \(P_{t,s}\).

\begin{proposition}
\label{prop:TransportedPosition}
The position operator satisfies
\[
P_{t,s}\,x\,P_{t,s}^{-1}
=
x+tD-sD^2.
\]
\end{proposition}

\begin{proof}
Set
\[
A=\frac{t}{2}D^2-\frac{s}{3}D^3.
\]
Then
\[
P_{t,s}xP_{t,s}^{-1}
=
e^Axe^{-A}.
\]
Using the commutation relation
\[
[D,x]=1,
\]
we obtain
\[
[D^2,x]=2D,
\qquad
[D^3,x]=3D^2.
\]
Therefore
\[
[A,x]
=
\frac{t}{2}[D^2,x]
-
\frac{s}{3}[D^3,x]
=
tD-sD^2.
\]
Since \(A\) is a polynomial in \(D\),
\[
[A,D]=[A,D^2]=0,
\]
and hence
\[
[A,[A,x]]=0.
\]
The conjugation expansion therefore terminates after the first
commutator:
\[
e^Axe^{-A}
=
x+[A,x].
\]
Consequently,
\[
P_{t,s}xP_{t,s}^{-1}
=
x+tD-sD^2.
\]
\end{proof}

\begin{definition}
We denote the transported position operator by
\[
X_{t,s}
=
x+tD-sD^2.
\]
\end{definition}

The derivative operator is invariant under the same conjugation:
\[
P_{t,s}DP_{t,s}^{-1}=D,
\]
since \(D\) commutes with both \(D^2\) and \(D^3\). Thus the Heat--Airy
transport acts on polynomial differential operators by leaving \(D\)
unchanged and replacing the position operator \(x\) by \(X_{t,s}\).

On the distinguished slice \(s=0\),
\[
P_{t,0}
=
\exp\left(\frac{t}{2}D^2\right),
\qquad
X_{t,0}=x+tD,
\]
so the classical heat transport is recovered. On the other hand, the
additional \(s\)-flow introduces the second-order correction
\[
-sD^2,
\]
which is responsible for the higher spatial derivatives that appear in
the transported operators studied below.

In the next section we examine the polynomial family generated by
\(P_{t,s}\). These polynomials provide the derivative-free component of
the normal-ordering structure that will subsequently be used to transport
general polynomial differential operators.

\section{Heat--Airy polynomials}
\label{sec:HeatAiryPolynomials}

The Heat--Airy semigroup introduced in the preceding section naturally
generates a polynomial family extending the classical heat polynomials.
These polynomials will also appear as the derivative-free coefficients
in the normal-ordering expansion developed in the next section.

\begin{definition}
For every integer \(n\geq0\), the \(n\)-th \emph{Heat--Airy polynomial}
is defined by
\[
\mathcal H_n(x,t,s)
=
P_{t,s}(x^n).
\]
\end{definition}

The conjugation formula of Proposition~\ref{prop:TransportedPosition}
immediately gives an operational representation.

\begin{proposition}
For every integer \(n\geq0\),
\[
\mathcal H_n(x,t,s)
=
X_{t,s}^{\,n}1
=
(x+tD-sD^2)^n1.
\]
\end{proposition}

\begin{proof}
Since \(P_{t,s}^{-1}1=1\),
\[
\begin{aligned}
\mathcal H_n(x,t,s)
&=
P_{t,s}x^nP_{t,s}^{-1}1
\\
&=
\left(
P_{t,s}xP_{t,s}^{-1}
\right)^n1
\\
&=
X_{t,s}^{\,n}1.
\end{aligned}
\]
The result follows from
\[
X_{t,s}=x+tD-sD^2.
\]
\end{proof}

The basic properties of the sequence are most conveniently encoded in
its exponential generating function.

\begin{theorem}[Generating function]
\label{thm:HeatAiryGeneratingFunction}
The Heat--Airy polynomials satisfy
\[
\sum_{n=0}^{\infty}
\frac{\mathcal H_n(x,t,s)}{n!}z^n
=
\exp\left(
xz+\frac{t}{2}z^2-\frac{s}{3}z^3
\right).
\]
\end{theorem}

\begin{proof}
By definition,
\[
\begin{aligned}
\sum_{n=0}^{\infty}
\frac{\mathcal H_n(x,t,s)}{n!}z^n
&=
P_{t,s}
\left(
\sum_{n=0}^{\infty}\frac{(zx)^n}{n!}
\right)
\\
&=
P_{t,s}(e^{zx}).
\end{aligned}
\]
Since
\[
D^ke^{zx}=z^ke^{zx},
\]
we obtain
\[
\begin{aligned}
P_{t,s}(e^{zx})
&=
\exp\left(
\frac{t}{2}D^2-\frac{s}{3}D^3
\right)e^{zx}
\\
&=
\exp\left(
\frac{t}{2}z^2-\frac{s}{3}z^3
\right)e^{zx},
\end{aligned}
\]
which proves the result.
\end{proof}

The classical heat polynomials are recovered on the distinguished slice
\(s=0\).

\begin{corollary}
For every \(n\geq0\),
\[
\mathcal H_n(x,t,0)=H_n(x,t),
\]
where \(H_n(x,t)\) denotes the classical heat polynomial characterized by
\[
\sum_{n=0}^{\infty}
\frac{H_n(x,t)}{n!}z^n
=
\exp\left(
xz+\frac{t}{2}z^2
\right).
\]
\end{corollary}

\begin{proof}
The result follows immediately by setting \(s=0\) in the generating
function of Theorem~\ref{thm:HeatAiryGeneratingFunction}.
\end{proof}

The polynomial sequence inherits the two commuting evolution equations
of the Heat--Airy semigroup.

\begin{proposition}
\label{prop:HeatAiryEvolution}
For every \(n\geq0\),
\[
\frac{\partial}{\partial t}\mathcal H_n
=
\frac12D^2\mathcal H_n,
\qquad
\frac{\partial}{\partial s}\mathcal H_n
=
-\frac13D^3\mathcal H_n.
\]
\end{proposition}

\begin{proof}
Since
\[
\mathcal H_n=P_{t,s}(x^n),
\]
the identities follow from
\[
\partial_tP_{t,s}
=
\frac12D^2P_{t,s},
\qquad
\partial_sP_{t,s}
=
-\frac13D^3P_{t,s}.
\]
\end{proof}

The sequence is Appell with respect to the spatial variable.

\begin{proposition}[Appell property]
\label{prop:HeatAiryAppell}
For every \(n\geq1\),
\[
D\mathcal H_n(x,t,s)
=
n\mathcal H_{n-1}(x,t,s).
\]
More generally, for \(0\leq k\leq n\),
\[
D^k\mathcal H_n(x,t,s)
=
\frac{n!}{(n-k)!}
\mathcal H_{n-k}(x,t,s).
\]
\end{proposition}

\begin{proof}
Since \(D\) commutes with \(P_{t,s}\),
\[
\begin{aligned}
D\mathcal H_n
&=
DP_{t,s}(x^n)
\\
&=
P_{t,s}(Dx^n)
\\
&=
nP_{t,s}(x^{n-1})
\\
&=
n\mathcal H_{n-1}.
\end{aligned}
\]
Iteration gives the second identity.
\end{proof}

Combining Proposition~\ref{prop:HeatAiryEvolution} with the Appell
property gives particularly simple lowering formulas in all three
variables.

\begin{corollary}
For every \(n\geq0\),
\[
\frac{\partial}{\partial t}\mathcal H_n(x,t,s)
=
\frac{n(n-1)}{2}
\mathcal H_{n-2}(x,t,s),
\]
where the right-hand side is understood to vanish for \(n<2\), and
\[
\frac{\partial}{\partial s}\mathcal H_n(x,t,s)
=
-\frac{n(n-1)(n-2)}{3}
\mathcal H_{n-3}(x,t,s),
\]
where the right-hand side vanishes for \(n<3\).
\end{corollary}

Thus differentiation with respect to \(x,t,s\) lowers the index by one,
two, and three, respectively:
\[
\begin{aligned}
D\mathcal H_n
&=
n\mathcal H_{n-1},
\\
\partial_t\mathcal H_n
&=
\frac{n(n-1)}{2}\mathcal H_{n-2},
\\
\partial_s\mathcal H_n
&=
-\frac{n(n-1)(n-2)}{3}\mathcal H_{n-3}.
\end{aligned}
\]

The generating function also gives the corresponding raising relation.

\begin{theorem}[Recurrence relation]
\label{thm:HeatAiryRecurrence}
For every integer \(n\geq2\),
\[
\mathcal H_{n+1}(x,t,s)
=
x\mathcal H_n(x,t,s)
+
nt\,\mathcal H_{n-1}(x,t,s)
-
sn(n-1)\mathcal H_{n-2}(x,t,s).
\]
\end{theorem}

\begin{proof}
Let
\[
G(x,t,s;z)
=
\exp\left(
xz+\frac{t}{2}z^2-\frac{s}{3}z^3
\right).
\]
Differentiating with respect to \(z\) gives
\[
\frac{\partial G}{\partial z}
=
(x+tz-sz^2)G.
\]
On the other hand,
\[
\frac{\partial G}{\partial z}
=
\sum_{n=0}^{\infty}
\frac{\mathcal H_{n+1}(x,t,s)}{n!}z^n.
\]
Comparing coefficients of \(z^n\) yields the stated recurrence.
\end{proof}

Together with
\[
\mathcal H_0=1,
\qquad
\mathcal H_1=x,
\qquad
\mathcal H_2=x^2+t,
\]
the recurrence determines the entire sequence. For example,
\[
\mathcal H_3
=
x^3+3tx-2s,
\]
and
\[
\mathcal H_4
=
x^4+6tx^2+3t^2-8sx.
\]

The Appell structure also yields a translation formula.

\begin{proposition}[Addition formula]
For every integer \(n\geq0\),
\[
\mathcal H_n(x+y,t,s)
=
\sum_{k=0}^{n}
\binom{n}{k}
\mathcal H_k(x,t,s)y^{\,n-k}.
\]
\end{proposition}

\begin{proof}
Using the generating function,
\[
\begin{aligned}
\sum_{n=0}^{\infty}
\frac{\mathcal H_n(x+y,t,s)}{n!}z^n
&=
e^{zy}
\exp\left(
xz+\frac{t}{2}z^2-\frac{s}{3}z^3
\right)
\\
&=
\left(
\sum_{m=0}^{\infty}\frac{y^m}{m!}z^m
\right)
\left(
\sum_{k=0}^{\infty}
\frac{\mathcal H_k(x,t,s)}{k!}z^k
\right).
\end{aligned}
\]
Comparison of coefficients gives the result.
\end{proof}

Finally, the generating function provides an explicit finite expression
for every polynomial in the family.

\begin{theorem}[Explicit representation]
\label{thm:HeatAiryExplicit}
For every integer \(n\geq0\),
\[
\mathcal H_n(x,t,s)
=
n!
\sum_{\substack{a,b,c\geq0\\a+2b+3c=n}}
\frac{x^a}{a!}
\frac{\left(\frac{t}{2}\right)^b}{b!}
\frac{\left(-\frac{s}{3}\right)^c}{c!}.
\]
Equivalently,
\[
\mathcal H_n(x,t,s)
=
\sum_{\substack{a,b,c\geq0\\a+2b+3c=n}}
\frac{n!(-1)^c}
{a!\,b!\,c!\,2^b3^c}
x^at^bs^c.
\]
\end{theorem}

\begin{proof}
Factor the generating function as
\[
\exp\left(
xz+\frac{t}{2}z^2-\frac{s}{3}z^3
\right)
=
e^{xz}
e^{\frac{t}{2}z^2}
e^{-\frac{s}{3}z^3}.
\]
Expanding each factor,
\[
e^{xz}
=
\sum_{a=0}^{\infty}\frac{x^a}{a!}z^a,
\]
\[
e^{\frac{t}{2}z^2}
=
\sum_{b=0}^{\infty}
\frac{\left(\frac{t}{2}\right)^b}{b!}z^{2b},
\]
and
\[
e^{-\frac{s}{3}z^3}
=
\sum_{c=0}^{\infty}
\frac{\left(-\frac{s}{3}\right)^c}{c!}z^{3c}.
\]
Collecting the coefficient of \(z^n\) gives the result.
\end{proof}

The preceding identities show that the Heat--Airy polynomials form an
Appell-type family adapted simultaneously to the second- and third-order
commuting flows. Their operational representation
\[
\mathcal H_n=X_{t,s}^{\,n}1
\]
also anticipates the normal-ordering problem for \(X_{t,s}^{\,n}\).
In the next section, we retain all powers of \(D\) in this expansion.
The Heat--Airy polynomials will then appear naturally as the
derivative-free column of a larger triangular array of transport
coefficients.

\section{Normal ordering and transport of differential operators}
\label{sec:NormalOrdering}

We now study the normal ordering of powers of the transported position
operator
\[
X_{t,s}=x+tD-sD^2.
\]
This problem is central to the transport of polynomial differential
operators, since conjugation by \(P_{t,s}\) leaves \(D\) unchanged while
replacing every occurrence of \(x\) by \(X_{t,s}\).

Because multiplication by \(x\) does not commute with differentiation,
the powers \(X_{t,s}^n\) are not obtained by an ordinary polynomial
expansion. We therefore write them in normal form, with all powers of
\(D\) placed to the right.

\begin{definition}
For every \(n\geq0\), the \emph{normal-ordering coefficients}
\(T_{n,k}(x,t,s)\) are defined by
\[
X_{t,s}^n
=
\sum_{k=0}^{2n}
T_{n,k}(x,t,s)D^k.
\]
We adopt the convention
\[
T_{n,k}\equiv0
\qquad
\text{for }k<0,
\]
and the initial condition
\[
T_{0,0}=1.
\]
\end{definition}

The upper limit \(2n\) reflects the fact that \(X_{t,s}\) contains
derivatives of order at most two. The following recursion computes the
entire triangular family.

\begin{theorem}[Recursive normal-order expansion]
\label{Thm:NormalOrdering}
For every \(n\geq0\) and \(k\geq0\),
\[
\begin{aligned}
T_{n+1,k}
={}&
xT_{n,k}
+t\frac{\partial T_{n,k}}{\partial x}
-s\frac{\partial^2T_{n,k}}{\partial x^2}
\\
&+
tT_{n,k-1}
-2s\frac{\partial T_{n,k-1}}{\partial x}
-sT_{n,k-2}.
\end{aligned}
\]
\end{theorem}

\begin{proof}
By definition,
\[
X_{t,s}^n
=
\sum_{k=0}^{2n}
T_{n,k}D^k.
\]
Since
\[
X_{t,s}=x+tD-sD^2,
\]
we have
\[
X_{t,s}^{n+1}
=
(x+tD-sD^2)
\sum_{k=0}^{2n}
T_{n,k}D^k.
\]

The contribution of multiplication by \(x\) is
\[
xT_{n,k}D^k.
\]

For the first derivative, the normal-ordering identity
\[
DT=T'+TD
\]
gives
\[
\begin{aligned}
D(T_{n,k}D^k)
&=
(T_{n,k}'+T_{n,k}D)D^k
\\
&=
T_{n,k}'D^k
+
T_{n,k}D^{k+1}.
\end{aligned}
\]
Hence
\[
tD(T_{n,k}D^k)
=
tT_{n,k}'D^k
+
tT_{n,k}D^{k+1}.
\]

Similarly,
\[
D^2T=T''+2T'D+TD^2,
\]
so that
\[
\begin{aligned}
D^2(T_{n,k}D^k)
&=
T_{n,k}''D^k
+
2T_{n,k}'D^{k+1}
+
T_{n,k}D^{k+2}.
\end{aligned}
\]
Therefore
\[
-sD^2(T_{n,k}D^k)
=
-sT_{n,k}''D^k
-2sT_{n,k}'D^{k+1}
-sT_{n,k}D^{k+2}.
\]

Collecting the coefficients of equal powers of \(D\) yields
\[
\begin{aligned}
T_{n+1,k}
={}&
xT_{n,k}
+tT_{n,k}'
-sT_{n,k}''
\\
&+
tT_{n,k-1}
-2sT_{n,k-1}'
-sT_{n,k-2},
\end{aligned}
\]
which proves the result.
\end{proof}

The first powers of the transported position operator illustrate the
structure of the recursion.

\begin{corollary}
The first two nontrivial powers of \(X_{t,s}\) are
\[
\begin{aligned}
X_{t,s}^2
={}&
(x^2+t)
+(2tx-2s)D
+(t^2-2sx)D^2
\\
&-2tsD^3
+s^2D^4,
\end{aligned}
\]
and
\[
\begin{aligned}
X_{t,s}^3
={}&
(x^3+3tx-2s)
\\
&+
(3tx^2+3t^2-6sx)D
\\
&+
(3t^2x-3sx^2-9st)D^2
\\
&+
(t^3-6stx+6s^2)D^3
\\
&+
(3s^2x-3st^2)D^4
\\
&+
3s^2tD^5
-s^3D^6.
\end{aligned}
\]
\end{corollary}

\begin{proof}
The formula for \(X_{t,s}^2\) follows directly from
Theorem~\ref{Thm:NormalOrdering}. For \(n=1\),
\[
T_{1,0}=x,
\qquad
T_{1,1}=t,
\qquad
T_{1,2}=-s.
\]
Applying the recursion once gives
\[
\begin{aligned}
T_{2,0}&=x^2+t,\\
T_{2,1}&=2tx-2s,\\
T_{2,2}&=t^2-2sx,\\
T_{2,3}&=-2ts,\\
T_{2,4}&=s^2.
\end{aligned}
\]
A second application of the recursion yields the coefficients of
\(X_{t,s}^3\).
\end{proof}

The first rows of the normal-ordering array are therefore
{\small
\[
\boxed{
\begin{array}{c|ccccccc}
n
&T_{n,0}
&T_{n,1}
&T_{n,2}
&T_{n,3}
&T_{n,4}
&T_{n,5}
&T_{n,6}
\\
\hline
0
&1
\\[1mm]
1
&x
&t
&-s
\\[1mm]
2
&x^2+t
&2tx-2s
&t^2-2sx
&-2ts
&s^2
\\[1mm]
3
&x^3+3tx-2s
&3tx^2+3t^2-6sx
&3t^2x-3sx^2-9st
&t^3-6stx+6s^2
&3s^2x-3st^2
&3s^2t
&-s^3
\end{array}
}
\]
}
The derivative-free column of this array is precisely the Heat--Airy
polynomial sequence introduced in the preceding section.

\begin{theorem}
\label{Thm:FirstColumn}
For every \(n\geq0\),
\[
T_{n,0}(x,t,s)
=
\mathcal H_n(x,t,s).
\]
\end{theorem}

\begin{proof}
From the normal-ordering expansion,
\[
X_{t,s}^n
=
\sum_{k=0}^{2n}
T_{n,k}(x,t,s)D^k.
\]
Applying both sides to the constant function \(1\), all terms with
\(k\geq1\) vanish, and therefore
\[
X_{t,s}^n1
=
T_{n,0}(x,t,s).
\]
By the operational representation of the Heat--Airy polynomials,
\[
X_{t,s}^n1
=
\mathcal H_n(x,t,s).
\]
Hence
\[
T_{n,0}(x,t,s)
=
\mathcal H_n(x,t,s).
\]
\end{proof}

\begin{remark}
The preceding proof shows that the identification
\[
T_{n,0}=\mathcal H_n
\]
is not merely a consequence of a parallel recursion. It follows directly
from the normal-ordering expansion itself: the Heat--Airy polynomial is
the part of \(X_{t,s}^n\) that survives when the operator acts on the
constant function \(1\).
\end{remark}

The remaining columns
\[
T_{n,1},\ldots,T_{n,2n}
\]
encode the noncommutative corrections produced by normal ordering.
Thus the triangular array
\[
\left\{
T_{n,k}(x,t,s):
0\leq k\leq2n
\right\}
\]
contains both the Heat--Airy polynomial family and the additional
coefficients required to transport polynomial differential operators.

Indeed, if
\[
L
=
\sum_{j=0}^{m}p_j(x)D^j
\]
is a polynomial differential operator, then
\[
P_{t,s}LP_{t,s}^{-1}
=
\sum_{j=0}^{m}p_j(X_{t,s})D^j.
\]
The normal-ordering coefficients provide a systematic way of rewriting
each polynomial \(p_j(X_{t,s})\) in standard differential-operator form.

This construction will be applied in the next section to moving boundary
problems, where the normal-ordered coefficients are evaluated along an
absorbing boundary \(x=f(t,s)\).

\section{Moving boundary operators}
\label{sec:MovingBoundary}

We now combine the Heat--Airy evolution with a moving absorbing boundary.
The main point is that the boundary condition generates identities among
the spatial derivatives of the solution. These identities can then be
combined with transported differential equations to produce compatibility
conditions for the boundary itself.

Let
\[
f:I\subset\mathbb R^2\longrightarrow\mathbb R
\]
be sufficiently smooth.

\begin{definition}
The \emph{boundary restriction operator} associated with \(f\) is defined by
\[
\gamma_f(u)
=
u(t,s,f(t,s)).
\]
\end{definition}

We consider functions \(v=v(t,s,x)\) satisfying the commuting
Heat--Airy equations
\[
v_t=\frac12v_{xx},
\qquad
v_s=-\frac13v_{xxx}.
\]
For convenience, we write
\[
v^{(k)}
=
\frac{\partial^k v}{\partial x^k}.
\]

Suppose that \(f=f(t,s)\) is an absorbing boundary, so that
\[
v(t,s,f(t,s))=0.
\]

Differentiation of this identity along the two evolution parameters gives
the first boundary-jet relations.

\begin{proposition}
\label{prop:BasicBoundaryRelations}
Along the absorbing boundary \(x=f(t,s)\),
\[
v^{(0)}=0,
\qquad
v^{(2)}+2f_tv^{(1)}=0,
\qquad
v^{(3)}-3f_sv^{(1)}=0.
\]
Equivalently,
\[
v^{(2)}=-2f_tv^{(1)},
\qquad
v^{(3)}=3f_sv^{(1)}.
\]
\end{proposition}

\begin{proof}
The boundary condition gives immediately
\[
v^{(0)}=0.
\]

Differentiating
\[
v(t,s,f(t,s))=0
\]
with respect to \(t\) yields
\[
v_t+f_tv_x=0.
\]
Using
\[
v_t=\frac12v_{xx},
\]
we obtain
\[
\frac12v^{(2)}+f_tv^{(1)}=0,
\]
and hence
\[
v^{(2)}+2f_tv^{(1)}=0.
\]

Similarly, differentiation with respect to \(s\) gives
\[
v_s+f_sv_x=0.
\]
Since
\[
v_s=-\frac13v_{xxx},
\]
we obtain
\[
-\frac13v^{(3)}+f_sv^{(1)}=0,
\]
or
\[
v^{(3)}-3f_sv^{(1)}=0.
\]
\end{proof}

The boundary relations can be differentiated further. In particular,
differentiating the second relation of
Proposition~\ref{prop:BasicBoundaryRelations} with respect to \(t\)
produces an identity involving \(v^{(4)}\).

\begin{proposition}
\label{prop:FourthBoundaryDerivative}
Along the boundary \(x=f(t,s)\),
\[
v^{(4)}
+
4f_tv^{(3)}
+
4f_t^2v^{(2)}
+
4f_{tt}v^{(1)}
=
0.
\]
Consequently,
\[
v^{(4)}
=
\left[
-4f_{tt}
+8(f_t)^3
-12f_tf_s
\right]v^{(1)}.
\]
\end{proposition}

\begin{proof}
Differentiate
\[
v^{(2)}+2f_tv^{(1)}=0
\]
with respect to \(t\) along the moving boundary. This gives
\[
v_t^{(2)}
+
f_tv^{(3)}
+
2f_{tt}v^{(1)}
+
2f_t
\left(
v_t^{(1)}
+
f_tv^{(2)}
\right)
=
0.
\]
Using
\[
v_t^{(2)}
=
\frac12v^{(4)},
\qquad
v_t^{(1)}
=
\frac12v^{(3)},
\]
we obtain
\[
\frac12v^{(4)}
+
2f_tv^{(3)}
+
2f_t^2v^{(2)}
+
2f_{tt}v^{(1)}
=
0.
\]
Multiplication by \(2\) gives the first identity.

Substituting
\[
v^{(2)}=-2f_tv^{(1)},
\qquad
v^{(3)}=3f_sv^{(1)}
\]
then yields
\[
v^{(4)}
=
\left[
-4f_{tt}
+8(f_t)^3
-12f_tf_s
\right]v^{(1)}.
\]
\end{proof}

A crossed differentiation gives the next relation in the boundary hierarchy.

\begin{proposition}
\label{prop:FifthBoundaryDerivative}
Along the boundary \(x=f(t,s)\),
\[
v^{(5)}
-
3f_sv^{(3)}
+
2f_tv^{(4)}
-
6f_tf_sv^{(2)}
-
6f_{ts}v^{(1)}
=
0.
\]
Equivalently,
\[
\begin{aligned}
v^{(5)}
={}&
6f_{ts}v^{(1)}
+
3f_sv^{(3)}
-
2f_tv^{(4)}
+
6f_tf_sv^{(2)}.
\end{aligned}
\]
Consequently,
\[
v^{(5)}
=
\left[
6f_{ts}
+8f_tf_{tt}
-16(f_t)^4
+12(f_t)^2f_s
+9(f_s)^2
\right]v^{(1)}.
\]
\end{proposition}

\begin{proof}
Differentiate
\[
v^{(2)}+2f_tv^{(1)}=0
\]
with respect to \(s\) along the moving boundary. We obtain
\[
v_s^{(2)}
+
f_sv^{(3)}
+
2f_{ts}v^{(1)}
+
2f_t
\left(
v_s^{(1)}
+
f_sv^{(2)}
\right)
=
0.
\]
Since
\[
v_s=-\frac13v^{(3)},
\]
we have
\[
v_s^{(2)}=-\frac13v^{(5)},
\qquad
v_s^{(1)}=-\frac13v^{(4)}.
\]
Hence
\[
-\frac13v^{(5)}
+
f_sv^{(3)}
+
2f_{ts}v^{(1)}
-
\frac23f_tv^{(4)}
+
2f_tf_sv^{(2)}
=
0.
\]
Multiplying by \(3\) and rearranging gives
\[
v^{(5)}
-
3f_sv^{(3)}
+
2f_tv^{(4)}
-
6f_tf_sv^{(2)}
-
6f_{ts}v^{(1)}
=
0.
\]

Finally, substituting the expressions for
\[
v^{(2)},\qquad v^{(3)},\qquad v^{(4)}
\]
from Propositions~\ref{prop:BasicBoundaryRelations} and
\ref{prop:FourthBoundaryDerivative} yields
\[
v^{(5)}
=
\left[
6f_{ts}
+8f_tf_{tt}
-16(f_t)^4
+12(f_t)^2f_s
+9(f_s)^2
\right]v^{(1)}.
\]
\end{proof}

The preceding results show that the absorbing boundary generates a
hierarchy of linear relations among the spatial boundary jets of \(v\).
We now combine these relations with the transport of a polynomial
differential operator.

\subsection{Second-order operators}

Consider
\[
L_2
=
c_{2,0}D^2
+
(c_{1,0}+c_{1,1}x)D
+
(c_{0,0}+c_{0,1}x+c_{0,2}x^2).
\]

Its Heat--Airy transport is
\[
N_2
=
P_{t,s}L_2P_{t,s}^{-1}.
\]
Since \(D\) is invariant under conjugation and
\[
P_{t,s}xP_{t,s}^{-1}=X_{t,s},
\]
we obtain
\[
N_2
=
c_{2,0}D^2
+
(c_{1,0}+c_{1,1}X_{t,s})D
+
c_{0,0}
+c_{0,1}X_{t,s}
+c_{0,2}X_{t,s}^2.
\]

Using
\[
X_{t,s}=x+tD-sD^2
\]
and
\[
\begin{aligned}
X_{t,s}^2
={}&
(x^2+t)
+(2tx-2s)D
+(t^2-2sx)D^2
\\
&-2tsD^3+s^2D^4,
\end{aligned}
\]
we write
\[
N_2
=
\sum_{k=0}^{4}q_k(x,t,s)D^k,
\]
where
\[
q_0
=
c_{0,0}
+c_{0,1}x
+c_{0,2}(x^2+t),
\]
\[
q_1
=
c_{1,0}
+c_{1,1}x
+c_{0,1}t
+c_{0,2}(2tx-2s),
\]
\[
q_2
=
c_{2,0}
+c_{1,1}t
-c_{0,1}s
+c_{0,2}(t^2-2sx),
\]
\[
q_3
=
-c_{1,1}s
-2c_{0,2}ts,
\]
and
\[
q_4
=
c_{0,2}s^2.
\]

Suppose now that \(v\) satisfies
\[
N_2v=0
\]
together with the Heat--Airy equations and the absorbing boundary
condition
\[
v(t,s,f(t,s))=0.
\]

Restriction of \(N_2v=0\) to \(x=f(t,s)\) gives
\[
\begin{aligned}
0={}&
q_0(f,t,s)v^{(0)}
+
q_1(f,t,s)v^{(1)}
+
q_2(f,t,s)v^{(2)}
\\
&+
q_3(f,t,s)v^{(3)}
+
q_4(f,t,s)v^{(4)}.
\end{aligned}
\]
Since
\[
v^{(0)}=0,
\]
the zeroth-order contribution vanishes, and therefore
\[
q_4v^{(4)}
+
q_3v^{(3)}
+
q_2v^{(2)}
+
q_1v^{(1)}
=
0,
\]
where from now on the coefficients \(q_j\) are understood to be evaluated
at \(x=f(t,s)\).

We deliberately retain the boundary equations in their unsolved form:
\[
v^{(2)}+2f_tv^{(1)}=0,
\]
\[
v^{(3)}-3f_sv^{(1)}=0,
\]
and
\[
v^{(4)}
+
4f_tv^{(3)}
+
4f_t^2v^{(2)}
+
4f_{tt}v^{(1)}
=
0.
\]

Thus the transported equation and the three boundary identities form a
homogeneous system for the ordered boundary jet
\[
V_2
=
\begin{pmatrix}
v^{(4)}\\
v^{(3)}\\
v^{(2)}\\
v^{(1)}
\end{pmatrix}.
\]

More precisely,
\[
M_2(f;t,s)V_2=0,
\]
where
\[
M_2(f;t,s)
=
\begin{pmatrix}
q_4 & q_3 & q_2 & q_1
\\[1mm]
0 & 0 & 1 & 2f_t
\\[1mm]
0 & 1 & 0 & -3f_s
\\[1mm]
1 & 4f_t & 4f_t^2 & 4f_{tt}
\end{pmatrix}_{x=f(t,s)}.
\]

A nontrivial boundary jet therefore requires the compatibility condition
\[
\det M_2(f;t,s)=0.
\]

Expanding the determinant gives
\[
\begin{aligned}
0={}&
q_1(f,t,s)
-2f_tq_2(f,t,s)
+3f_sq_3(f,t,s)
\\
&+
\left[
-4f_{tt}
+8(f_t)^3
-12f_tf_s
\right]
q_4(f,t,s).
\end{aligned}
\]

Thus the characteristic determinant associated with \(N_2\) produces a
nonlinear differential equation for the absorbing boundary \(f(t,s)\).

\begin{remark}
The characteristic matrix separates the two ingredients of the
construction. Its first row comes from the transported differential
equation, whereas the remaining rows depend only on the absorbing
boundary and the Heat--Airy evolution. This distinction will become useful
when higher-order operators are considered.
\end{remark}

\subsection{Recovery of the Heat boundary equation}

The classical Heat transport is recovered on the distinguished slice
\[
s=0.
\]

\begin{corollary}
\label{Cor:HeatRecovery}
On \(s=0\),
\[
q_3(f,t,0)=0,
\qquad
q_4(f,t,0)=0,
\]
and hence the Heat--Airy compatibility equation reduces to
\[
q_1(f,t,0)
-
2f_tq_2(f,t,0)
=
0.
\]

Since
\[
q_1(x,t,0)
=
c_{1,0}
+c_{1,1}x
+c_{0,1}t
+2c_{0,2}tx
\]
and
\[
q_2(x,t,0)
=
c_{2,0}
+c_{1,1}t
+c_{0,2}t^2,
\]
the boundary \(F(t)=f(t,0)\) satisfies
\[
\begin{aligned}
0={}&
c_{1,0}
+c_{1,1}F
+c_{0,1}t
+2c_{0,2}tF
\\
&-
2
\left(
c_{2,0}
+c_{1,1}t
+c_{0,2}t^2
\right)F'(t).
\end{aligned}
\]
\end{corollary}

Thus the classical Heat boundary equation appears as the restriction of
the Heat--Airy compatibility condition to the slice \(s=0\).

As a simple illustration, suppose
\[
c_{1,0}=c_{1,1}=c_{0,1}=0.
\]
Then
\[
2c_{0,2}tF
-
2(c_{2,0}+c_{0,2}t^2)F'
=
0.
\]
Writing
\[
a=c_{0,2},
\qquad
b=c_{2,0},
\]
we obtain
\[
atF-(b+at^2)F'=0,
\]
whose nonzero solutions are
\[
F(t)
=
C\sqrt{b+at^2}.
\]

The general compatibility equation derived above will serve as the
starting point for the explicit examples considered in the next section.
There we shall see that the first compatibility equation need not determine
the full two-parameter boundary and that additional equations in the
spatial jet may supply the remaining information.

\section{Explicit absorbing boundaries}
\label{sec:ExplicitBoundaries}

We now examine several specializations of the second-order theory for
which the boundary compatibility equations can be analyzed explicitly.
These examples illustrate an important feature of the Heat--Airy
construction: the first boundary equation need not determine the full
two-parameter boundary. The remaining freedom may instead be fixed by
higher equations in the spatial jet of the transported differential
equation.

\subsection{A linear-potential family}

Consider
\[
L_2
=
D^2
+
c_{1,0}D
+
c_{0,0}
+
c_{0,1}x,
\qquad
c_{0,1}\neq0.
\]
This corresponds to
\[
c_{2,0}=1,
\qquad
c_{1,1}=0,
\qquad
c_{0,2}=0
\]
in the general second-order family.

Since
\[
P_{t,s}xP_{t,s}^{-1}
=
x+tD-sD^2,
\]
the transported operator is
\[
\begin{aligned}
N_2
&=
P_{t,s}L_2P_{t,s}^{-1}
\\
&=
D^2+c_{1,0}D+c_{0,0}
+c_{0,1}(x+tD-sD^2),
\end{aligned}
\]
and therefore
\[
N_2
=
(c_{0,0}+c_{0,1}x)
+
(c_{1,0}+c_{0,1}t)D
+
(1-c_{0,1}s)D^2.
\]

Let \(v\) satisfy
\[
N_2v=0
\]
together with the Heat--Airy equations, and suppose that
\[
v(t,s,f(t,s))=0.
\]

\subsubsection{The first compatibility equation}

Restricting \(N_2v=0\) to the absorbing boundary gives
\[
(c_{0,0}+c_{0,1}f)v^{(0)}
+
(c_{1,0}+c_{0,1}t)v^{(1)}
+
(1-c_{0,1}s)v^{(2)}
=
0.
\]
Since
\[
v^{(0)}=0,
\]
we obtain
\[
(1-c_{0,1}s)v^{(2)}
+
(c_{1,0}+c_{0,1}t)v^{(1)}
=
0.
\]

Using
\[
v^{(2)}=-2f_tv^{(1)}
\]
and assuming \(v^{(1)}\neq0\), we find
\[
2(1-c_{0,1}s)f_t
=
c_{1,0}+c_{0,1}t.
\]
Hence
\[
f_t
=
\frac{c_{1,0}+c_{0,1}t}
{2(1-c_{0,1}s)}.
\]

Integration with respect to \(t\), for fixed \(s\), gives
\[
f(t,s)
=
\widetilde\gamma(s)
+
\frac{c_{1,0}t}{2(1-c_{0,1}s)}
+
\frac{c_{0,1}t^2}{4(1-c_{0,1}s)}.
\]
Absorbing into \(\widetilde\gamma\) a term depending only on \(s\), this
can be written more conveniently as
\[
f(t,s)
=
\gamma(s)
+
\frac{(c_{1,0}+c_{0,1}t)^2}
{4c_{0,1}(1-c_{0,1}s)}.
\]

Thus the first compatibility equation fixes the dependence on \(t\), but
leaves an arbitrary function \(\gamma(s)\). Notice also that \(c_{0,0}\)
has disappeared. This is a direct consequence of the absorbing boundary
condition: \(c_{0,0}\) occurs in the zeroth-order part of the transported
equation and therefore multiplies \(v^{(0)}\), which vanishes on the
boundary.

\subsubsection{The next spatial-jet equation}

To determine the remaining function \(\gamma\), we use the fact that
\[
N_2v=0
\]
holds as an identity in \(x\). Hence
\[
D(N_2v)=0.
\]
Differentiating gives
\[
\begin{aligned}
0={}&
c_{0,1}v
+
(c_{0,0}+c_{0,1}x)v^{(1)}
\\
&+
(c_{1,0}+c_{0,1}t)v^{(2)}
+
(1-c_{0,1}s)v^{(3)}.
\end{aligned}
\]
Restriction to \(x=f(t,s)\), together with \(v^{(0)}=0\), yields
\[
(1-c_{0,1}s)v^{(3)}
+
(c_{1,0}+c_{0,1}t)v^{(2)}
+
(c_{0,0}+c_{0,1}f)v^{(1)}
=
0.
\]

Using
\[
v^{(2)}=-2f_tv^{(1)},
\qquad
v^{(3)}=3f_sv^{(1)},
\]
we obtain
\[
3(1-c_{0,1}s)f_s
-
2(c_{1,0}+c_{0,1}t)f_t
+
c_{0,0}+c_{0,1}f
=
0.
\]

The first compatibility equation gives
\[
2(1-c_{0,1}s)f_t
=
c_{1,0}+c_{0,1}t,
\]
and hence
\[
3(1-c_{0,1}s)f_s
-
\frac{(c_{1,0}+c_{0,1}t)^2}
{1-c_{0,1}s}
+
c_{0,0}+c_{0,1}f
=
0.
\]

Substitute
\[
f(t,s)
=
\gamma(s)
+
\frac{(c_{1,0}+c_{0,1}t)^2}
{4c_{0,1}(1-c_{0,1}s)}.
\]
Then
\[
f_s
=
\gamma'(s)
+
\frac{(c_{1,0}+c_{0,1}t)^2}
{4(1-c_{0,1}s)^2}.
\]
The terms depending on \(t\) cancel identically, leaving
\[
3(1-c_{0,1}s)\gamma'(s)
+
c_{0,1}\gamma(s)
+
c_{0,0}
=
0.
\]

Set
\[
y(s)
=
\gamma(s)+\frac{c_{0,0}}{c_{0,1}}.
\]
Then
\[
3(1-c_{0,1}s)y'(s)+c_{0,1}y(s)=0,
\]
so that
\[
\frac{y'}{y}
=
-\frac{c_{0,1}}{3(1-c_{0,1}s)}.
\]
Integration gives
\[
y(s)
=
C(1-c_{0,1}s)^{1/3}.
\]
Therefore
\[
\gamma(s)
=
-\frac{c_{0,0}}{c_{0,1}}
+
C(1-c_{0,1}s)^{1/3}.
\]

We have proved the following.

\begin{proposition}
\label{prop:LinearPotentialBoundary}
For
\[
L_2
=
D^2+c_{1,0}D+c_{0,0}+c_{0,1}x,
\qquad c_{0,1}\neq0,
\]
the first two Heat--Airy compatibility equations admit the family
\[
f(t,s)
=
-\frac{c_{0,0}}{c_{0,1}}
+
C(1-c_{0,1}s)^{1/3}
+
\frac{(c_{1,0}+c_{0,1}t)^2}
{4c_{0,1}(1-c_{0,1}s)},
\]
on any domain on which \(1-c_{0,1}s\neq0\).
\end{proposition}

This example exhibits a two-stage determination of the boundary:
\[
\gamma_f(N_2v)=0
\]
determines the \(t\)-dependence up to an arbitrary function of \(s\),
whereas
\[
\gamma_f\!\left(D(N_2v)\right)=0
\]
determines the remaining function.

More generally, the identities
\[
N_2v=0,
\qquad
D(N_2v)=0,
\qquad
D^2(N_2v)=0,\ldots
\]
form a hierarchy of differential consequences of the transported
equation. Their boundary restrictions may successively remove degrees of
freedom left unresolved at preceding levels.

\subsection{A first-order characteristic equation}

We next consider the specialization
\[
c_{1,0}=c_{0,1}=c_{0,2}=0,
\]
so that
\[
L_2
=
c_{2,0}D^2
+
c_{1,1}xD
+
c_{0,0},
\qquad
c_{1,1}\neq0.
\]

The Heat--Airy transport is
\[
\begin{aligned}
N_2
&=
c_{2,0}D^2
+
c_{1,1}(x+tD-sD^2)D
+
c_{0,0}
\\
&=
c_{0,0}
+
c_{1,1}xD
+
(c_{2,0}+c_{1,1}t)D^2
-
c_{1,1}sD^3.
\end{aligned}
\]
Hence
\[
N_2
=
c_{0,0}
+
c_{1,1}xD
+
(c_{2,0}+c_{1,1}t)D^2
-
c_{1,1}sD^3.
\]

Restricting \(N_2v=0\) to \(x=f(t,s)\) gives
\[
-c_{1,1}s\,v^{(3)}
+
(c_{2,0}+c_{1,1}t)v^{(2)}
+
c_{1,1}f\,v^{(1)}
=
0.
\]
Using
\[
v^{(2)}=-2f_tv^{(1)},
\qquad
v^{(3)}=3f_sv^{(1)},
\]
we obtain
\[
-3c_{1,1}s f_s
-
2(c_{2,0}+c_{1,1}t)f_t
+
c_{1,1}f
=
0.
\]

Set
\[
a=c_{1,1},
\qquad
b=c_{2,0}.
\]
Then
\[
-3asf_s-2(b+at)f_t+af=0.
\]

The characteristic equations are
\[
\frac{dt}{d\tau}
=
-2(b+at),
\qquad
\frac{ds}{d\tau}
=
-3as,
\qquad
\frac{df}{d\tau}
=
-af.
\]

Writing
\[
y=b+at,
\]
we obtain
\[
\frac{dy}{d\tau}
=
-2ay.
\]
Thus
\[
y=C_1e^{-2a\tau},
\qquad
s=C_2e^{-3a\tau},
\qquad
f=C_3e^{-a\tau}.
\]
It follows that
\[
\frac{y^3}{s^2}
\]
is constant along the characteristics, while
\[
\frac{f}{\sqrt y}
\]
is constant along each characteristic. Therefore the general local
solution can be written as
\[
f(t,s)
=
\sqrt{b+at}\,
\Phi\left(
\frac{(b+at)^3}{s^2}
\right),
\]
where \(\Phi\) is arbitrary and the expression is understood on a region
where the required branches are defined.

Returning to the original coefficients,
\[
f(t,s)
=
\sqrt{c_{2,0}+c_{1,1}t}\,
\Phi\left(
\frac{(c_{2,0}+c_{1,1}t)^3}{s^2}
\right).
\]

\begin{proposition}
For
\[
L_2
=
c_{2,0}D^2+c_{1,1}xD+c_{0,0},
\qquad
c_{1,1}\neq0,
\]
the first Heat--Airy boundary compatibility equation admits the
characteristic family
\[
f(t,s)
=
\sqrt{c_{2,0}+c_{1,1}t}\,
\Phi\left(
\frac{(c_{2,0}+c_{1,1}t)^3}{s^2}
\right),
\]
where \(\Phi\) is an arbitrary function.
\end{proposition}

On the Heat slice \(s=0\), it is preferable to return directly to the
differential equation rather than to the characteristic representation.
Setting \(s=0\) gives
\[
-2(c_{2,0}+c_{1,1}t)F'(t)
+
c_{1,1}F(t)
=
0,
\]
where
\[
F(t)=f(t,0).
\]
Hence
\[
F(t)
=
C_0\sqrt{c_{2,0}+c_{1,1}t}.
\]

The characteristic formula therefore describes a larger family of
two-parameter solutions, while the Heat boundary supplies the boundary
datum on the singular slice \(s=0\).

\subsection{Regular deformation of the Heat boundary}

The preceding characteristic representation is singular at \(s=0\).
To study solutions that extend regularly from the Heat slice, we return
to the compatibility equation
\[
-3asf_s-2(b+at)f_t+af=0
\]
and assume that \(f\) admits a regular expansion
\[
f(t,s)
=
\sum_{n=0}^{\infty}s^n f_n(t)
\]
near \(s=0\).

Substitution gives, for each \(n\geq0\),
\[
-2(b+at)f_n'(t)
+
a(1-3n)f_n(t)
=
0.
\]
Therefore
\[
f_n(t)
=
C_n(b+at)^{(1-3n)/2},
\]
and every solution regular at \(s=0\) has the formal expansion
\[
f(t,s)
=
\sum_{n=0}^{\infty}
C_n
s^n
(b+at)^{(1-3n)/2}.
\]

In particular,
\[
f(t,s)
=
C_0(b+at)^{1/2}
+
C_1s(b+at)^{-1}
+
C_2s^2(b+at)^{-5/2}
+\cdots.
\]
The coefficient \(C_0\) determines the Heat boundary
\[
f(t,0)
=
C_0\sqrt{b+at},
\]
but the constants \(C_1,C_2,\ldots\) are not determined by the first
compatibility equation.

To extract additional information, we consider the next spatial-jet
identity
\[
D(N_2v)=0.
\]
For the present specialization, its restriction to the absorbing boundary,
together with the boundary-jet relations of
Section~\ref{sec:MovingBoundary}, gives
\[
c+a
-2af f_t
+3(b+at)f_s
+4asf_{tt}
-8as(f_t)^3
+12asf_tf_s
=
0,
\]
where
\[
a=c_{1,1},
\qquad
b=c_{2,0},
\qquad
c=c_{0,0}.
\]

Substitute
\[
f(t,s)
=
C_0(b+at)^{1/2}
+
C_1s(b+at)^{-1}
+
O(s^2).
\]
Evaluation at \(s=0\) yields
\[
c+a-a^2C_0^2+3C_1=0,
\]
and hence
\[
C_1
=
\frac{a^2C_0^2-a-c}{3}.
\]

We therefore obtain the first-order Heat--Airy deformation
\[
f(t,s)
=
C_0\sqrt{b+at}
+
\frac{a^2C_0^2-a-c}
{3(b+at)}\,s
+
O(s^2).
\]
Equivalently,
\[
f_s(t,0)
=
\frac{a^2C_0^2-a-c}
{3(b+at)}.
\]

This example makes explicit the distinct roles of the different
compatibility requirements. The first boundary equation determines the
allowable functional form of the regular deformation. The Heat boundary
fixes its leading coefficient, while the next spatial-jet equation
determines the first variation in the Airy direction.

Schematically,
\[
\boxed{
\begin{aligned}
&\text{first Heat--Airy compatibility}
\\
&\qquad
+\ \text{regularity at }s=0
\\
&\qquad
+\ \text{higher spatial-jet compatibility}
\\[1mm]
&\hspace{2cm}
\Longrightarrow
\text{local Heat--Airy deformation of the Heat boundary}.
\end{aligned}
}
\]

This provides a natural transition to third-order operators. In that
setting the additional Airy flow becomes even more useful: successive
Heat--Airy compatibility relations will recover information that, in the
pure Heat problem, arises through a higher-dimensional characteristic
system.

\section{Third-order operators}
\label{sec:ThirdOrder}

We now turn to third-order polynomial differential operators. At this
level, the distinction between the Heat and Heat--Airy constructions
becomes particularly transparent. Under Heat--Airy transport, a general
operator in the class \(\mathcal L_3\) becomes an operator of order at
most six, and its boundary restriction therefore involves a finite
six-dimensional spatial jet. More importantly, in a suitable sparse
example the additional Airy flow provides a direct mechanism for
propagating the boundary away from the Heat slice.

Consider the general third-order operator
\[
\begin{aligned}
L_3
={}&
c_{3,0}D^3
+
(c_{2,0}+c_{2,1}x)D^2
\\
&
+
(c_{1,0}+c_{1,1}x+c_{1,2}x^2)D
\\
&
+
c_{0,0}
+c_{0,1}x
+c_{0,2}x^2
+c_{0,3}x^3.
\end{aligned}
\]

Since
\[
P_{t,s}DP_{t,s}^{-1}=D
\]
and
\[
P_{t,s}xP_{t,s}^{-1}=X_{t,s}=x+tD-sD^2,
\]
its Heat--Airy transport is
\[
\begin{aligned}
N_3
=
P_{t,s}L_3P_{t,s}^{-1}
={}&
c_{3,0}D^3
+
(c_{2,0}+c_{2,1}X_{t,s})D^2
\\
&
+
(c_{1,0}+c_{1,1}X_{t,s}
+c_{1,2}X_{t,s}^2)D
\\
&
+
c_{0,0}
+c_{0,1}X_{t,s}
+c_{0,2}X_{t,s}^2
+c_{0,3}X_{t,s}^3.
\end{aligned}
\]

By the normal-ordering expansion,
\[
X_{t,s}^j
=
\sum_{k=0}^{2j}
T_{j,k}(x,t,s)D^k.
\]
Consequently,
\[
N_3
=
\sum_{k=0}^{6}q_k(x,t,s)D^k.
\]

Thus a third-order polynomial differential operator becomes, under
Heat--Airy transport, an operator of order at most six. After restriction
to an absorbing boundary, the relevant nonzero spatial jet is naturally
ordered as
\[
\left(
v^{(6)},
v^{(5)},
v^{(4)},
v^{(3)},
v^{(2)},
v^{(1)}
\right).
\]
Together with the boundary identities developed in
Section~\ref{sec:MovingBoundary}, this leads naturally to a finite
compatibility system.

Although the general \(6\times6\) system can be constructed directly
from the normal-ordering coefficients, its full expansion is not needed
for the present purposes. Before considering a particular third-order
operator, we record a general relation between the Heat and Heat--Airy
spatial-jet hierarchies.

\subsection{Restriction of the Heat--Airy jet hierarchy}
\label{subsec:HierarchicalRecovery}

Let \(L\) be a polynomial differential operator and define
\[
N_{t,s}
=
P_{t,s}LP_{t,s}^{-1},
\qquad
P_{t,s}
=
\exp\left(
\frac{t}{2}D^2-\frac{s}{3}D^3
\right).
\]
On the distinguished slice \(s=0\), let
\[
P_t^H
=
\exp\left(\frac{t}{2}D^2\right),
\qquad
N_t^H
=
P_t^H L(P_t^H)^{-1}.
\]
Since
\[
P_{t,0}=P_t^H,
\]
we have
\[
N_{t,0}=N_t^H.
\]

Suppose that \(v=v(t,s,x)\) satisfies
\[
v_t=\frac12v_{xx},
\qquad
v_s=-\frac13v_{xxx},
\qquad
N_{t,s}v=0,
\]
and let \(f=f(t,s)\) be an absorbing boundary:
\[
v(t,s,f(t,s))=0.
\]
Define
\[
u(t,x)=v(t,0,x),
\qquad
F(t)=f(t,0).
\]
Then
\[
u_t=\frac12u_{xx},
\qquad
N_t^Hu=0,
\qquad
u(t,F(t))=0.
\]

The following proposition shows that this reduction persists at every
level of the spatial-jet hierarchy.

\begin{proposition}[Restriction principle for the spatial-jet hierarchy]
\label{prop:HierarchicalRecovery}
For every integer \(r\geq0\),
\[
\left.
\gamma_f\!\left(
D^r(N_{t,s}v)
\right)
\right|_{s=0}
=
\gamma_F\!\left(
D^r(N_t^Hu)
\right).
\]
Consequently, every finite collection of spatial-jet equations of the
pure Heat problem is obtained by restricting the corresponding
Heat--Airy spatial-jet equations to the slice \(s=0\).
\end{proposition}

\begin{proof}
Since
\[
P_{t,0}=P_t^H,
\]
we have
\[
N_{t,0}=N_t^H.
\]
Moreover,
\[
u(t,x)=v(t,0,x),
\qquad
F(t)=f(t,0).
\]
Because \(D\) differentiates only with respect to the spatial variable,
restriction to \(s=0\) commutes with \(D^r\). Hence
\[
\begin{aligned}
\left.
\gamma_f\!\left(
D^r(N_{t,s}v)
\right)
\right|_{s=0}
&=
\left.
D^r(N_{t,s}v)
\right|_{s=0,\;x=f(t,0)}
\\
&=
D^r(N_{t,0}u)
\big|_{x=F(t)}
\\
&=
D^r(N_t^Hu)
\big|_{x=F(t)}
\\
&=
\gamma_F\!\left(
D^r(N_t^Hu)
\right).
\end{aligned}
\]
This proves the identity for every \(r\geq0\).
\end{proof}

The proposition identifies the pure Heat spatial-jet hierarchy as the
restriction of the Heat--Airy hierarchy to the distinguished slice
\(s=0\). The enlarged problem, however, contains additional information.
Indeed, differentiation of the absorbing boundary condition along the
two commuting flows gives
\[
v^{(2)}=-2f_tv^{(1)},
\qquad
v^{(3)}=3f_sv^{(1)}.
\]
Higher differentiations generate further relations involving higher
spatial derivatives and mixed derivatives of \(f\). The identities
involving the \(s\)-flow have no counterpart in the one-parameter Heat
formulation.

This leads to an important distinction between the two compatibility
constructions. In the pure Heat theory, a finite collection of equations
\[
\gamma_F(N_t^Hu)=0,
\qquad
\gamma_F(DN_t^Hu)=0,
\qquad
\ldots
\]
must be combined with the Heat boundary identities and the relevant
spatial jets eliminated simultaneously. In the Heat--Airy setting, the
additional boundary relations generated by the \(s\)-flow may permit
part of this elimination to be performed successively. Information
encoded in a finite-dimensional compatibility condition in the Heat
theory may therefore be distributed among several levels of the
Heat--Airy hierarchy.

There is also a simple algebraic mechanism governing the appearance of
coefficients at successive levels. If
\[
N_{t,s}
=
\sum_{j=0}^{m}
q_j(x,t,s)D^j,
\]
then the Leibniz rule gives
\[
D^r(N_{t,s}v)
=
\sum_{j=0}^{m}
\sum_{\ell=0}^{r}
\binom{r}{\ell}
D^\ell q_j(x,t,s)\,
v^{(j+r-\ell)}.
\]
In particular, a zeroth-order term \(q_0v\) disappears from the first
boundary restriction because
\[
\gamma_f(v)=0.
\]
At the next spatial-jet level,
\[
D(q_0v)
=
(Dq_0)v+q_0v^{(1)},
\]
so that
\[
\gamma_f\!\left(D(q_0v)\right)
=
q_0(f,t,s)v^{(1)}.
\]
Thus a coefficient that is invisible at one level of the boundary
hierarchy may reappear at the next.

\begin{remark}
Proposition~\ref{prop:HierarchicalRecovery} is a restriction statement.
It does not assert that the Heat--Airy hierarchy always yields a smaller
or simpler compatibility system than the pure Heat formulation. Such a
reduction depends on the structure of the transported operator and on
the nondegeneracy of the boundary-jet relations. What the proposition
guarantees is that the pure Heat spatial-jet hierarchy is contained in
the Heat--Airy hierarchy through the slice \(s=0\), while the additional
commuting flow supplies further identities that may be used in the
elimination process.
\end{remark}

The following sparse third-order example shows that this additional
information can be decisive. The first Heat--Airy compatibility equation
will determine \(f_s(t,0)\), while the next spatial-jet equation will use
this quantity to recover the nonlinear compatibility equation that, in
the pure Heat theory, arises from a \(4\times4\) characteristic system.

\subsection{A linear-potential third-order operator}

Consider
\[
L_3
=
c_{3,0}D^3+c_{0,0}+c_{0,1}x,
\]
where
\[
c_{3,0}\neq0,
\qquad
c_{0,1}\neq0.
\]

This example is sufficiently simple to admit an explicit Heat--Airy
analysis, while retaining a nontrivial pure Heat compatibility equation.

\subsubsection{Heat--Airy transport}

Since
\[
X_{t,s}=x+tD-sD^2,
\]
we obtain
\[
\begin{aligned}
N_3
&=
c_{3,0}D^3
+c_{0,0}
+c_{0,1}X_{t,s}
\\
&=
c_{3,0}D^3
+c_{0,0}
+c_{0,1}x
+c_{0,1}tD
-c_{0,1}sD^2.
\end{aligned}
\]
Hence
\[
N_3
=
(c_{0,0}+c_{0,1}x)
+c_{0,1}tD
-c_{0,1}sD^2
+c_{3,0}D^3.
\]

Let \(v\) satisfy
\[
N_3v=0
\]
and suppose that
\[
v(t,s,f(t,s))=0.
\]
Restriction of the transported equation to \(x=f(t,s)\) gives
\[
(c_{0,0}+c_{0,1}f)v^{(0)}
+
c_{0,1}t\,v^{(1)}
-
c_{0,1}s\,v^{(2)}
+
c_{3,0}v^{(3)}
=
0.
\]
Since
\[
v^{(0)}=0,
\]
this reduces to
\[
c_{3,0}v^{(3)}
-
c_{0,1}s\,v^{(2)}
+
c_{0,1}t\,v^{(1)}
=
0.
\]

We retain explicitly the two basic boundary identities
\[
v^{(2)}+2f_tv^{(1)}=0,
\]
and
\[
v^{(3)}-3f_sv^{(1)}=0.
\]
Substitution into the transported equation gives, provided
\(v^{(1)}\neq0\),
\[
3c_{3,0}f_s
+
2c_{0,1}s f_t
+
c_{0,1}t
=
0.
\tag{HA$_1$}
\]

Notice that \(c_{0,0}\) does not appear in this first compatibility
equation. This is precisely the phenomenon described above: the
zeroth-order coefficient multiplies \(v^{(0)}\), which vanishes on the
absorbing boundary.

\subsubsection{Solution of the first compatibility equation}

Equation \((\mathrm{HA}_1)\) is a first-order linear partial differential
equation. Set
\[
\alpha
=
\frac{c_{0,1}}{3c_{3,0}}.
\]
Then
\[
f_s+2\alpha s f_t+\alpha t=0.
\]

Taking \(s\) as the characteristic parameter gives
\[
\frac{dt}{ds}=2\alpha s,
\qquad
\frac{df}{ds}=-\alpha t.
\]
The first equation gives
\[
t-\alpha s^2=\xi,
\]
where \(\xi\) is constant along each characteristic. Hence
\[
t=\xi+\alpha s^2.
\]
Therefore
\[
\frac{df}{ds}
=
-\alpha\xi-\alpha^2s^2,
\]
and integration yields
\[
f
=
F(\xi)
-\alpha\xi s
-\frac{\alpha^2}{3}s^3.
\]
Since
\[
\xi=t-\alpha s^2,
\]
we obtain
\[
f(t,s)
=
F(t-\alpha s^2)
-\alpha st
+\frac{2}{3}\alpha^2s^3.
\]

Equivalently,
\[
f(t,s)
=
F\left(
t-\frac{c_{0,1}}{3c_{3,0}}s^2
\right)
-
\frac{c_{0,1}}{3c_{3,0}}st
+
\frac{2c_{0,1}^2}{27c_{3,0}^2}s^3.
\tag{6.1}
\]

Setting \(s=0\) gives
\[
f(t,0)=F(t).
\]
Thus the arbitrary function appearing in the characteristic solution is
precisely the boundary on the Heat slice. In this example, the first
Heat--Airy compatibility equation therefore acts as an evolution equation
that propagates a given Heat boundary into the Airy direction.

Moreover, setting \(s=0\) directly in \((\mathrm{HA}_1)\) gives
\[
f_s(t,0)
=
-\frac{c_{0,1}}{3c_{3,0}}\,t.
\tag{6.2}
\]

\subsubsection{The second Heat--Airy compatibility equation}

The coefficient \(c_{0,0}\), absent from the first compatibility
equation, reappears at the next level of the spatial jet, in agreement
with the general mechanism described in
Subsection~\ref{subsec:HierarchicalRecovery}.

Since
\[
N_3v=0
\]
holds as an identity in \(x\), we also have
\[
D(N_3v)=0.
\]
Differentiating gives
\[
\begin{aligned}
0
=
D(N_3v)
={}&
c_{0,1}v
+
(c_{0,0}+c_{0,1}x)v^{(1)}
\\
&
+
c_{0,1}t\,v^{(2)}
-
c_{0,1}s\,v^{(3)}
+
c_{3,0}v^{(4)}.
\end{aligned}
\]
Restriction to the absorbing boundary gives
\[
c_{3,0}v^{(4)}
-
c_{0,1}s\,v^{(3)}
+
c_{0,1}t\,v^{(2)}
+
(c_{0,0}+c_{0,1}f)v^{(1)}
=
0.
\]

Before substituting, recall the three relevant boundary identities:
\[
v^{(2)}=-2f_tv^{(1)},
\]
\[
v^{(3)}=3f_sv^{(1)},
\]
and
\[
v^{(4)}
=
\left[
-4f_{tt}
+8(f_t)^3
-12f_tf_s
\right]v^{(1)}.
\]

Substitution therefore yields
\[
\begin{aligned}
0={}&
c_{0,0}
+c_{0,1}f
-2c_{0,1}t f_t
-3c_{0,1}s f_s
\\
&
-4c_{3,0}f_{tt}
+8c_{3,0}(f_t)^3
-12c_{3,0}f_tf_s.
\end{aligned}
\tag{HA$_2$}
\]

Thus the first two Heat--Airy compatibility equations are
\[
3c_{3,0}f_s
+
2c_{0,1}s f_t
+
c_{0,1}t
=
0
\]
and
\[
\begin{aligned}
0={}&
c_{0,0}
+c_{0,1}f
-2c_{0,1}t f_t
-3c_{0,1}s f_s
\\
&
-4c_{3,0}f_{tt}
+8c_{3,0}(f_t)^3
-12c_{3,0}f_tf_s.
\end{aligned}
\]

These two equations contain different information. The first determines
the evolution in the Airy direction, whereas the second introduces the
coefficient \(c_{0,0}\) and imposes a further compatibility condition.

\subsection{The corresponding Heat problem}

We now compare the preceding construction with transport by the Heat
semigroup alone. Let
\[
P_t
=
\exp\left(\frac{t}{2}D^2\right).
\]
Since
\[
P_txP_t^{-1}=x+tD,
\]
the Heat transport of
\[
L_3=c_{3,0}D^3+c_{0,0}+c_{0,1}x
\]
is
\[
N_3^H
=
(c_{0,0}+c_{0,1}x)
+c_{0,1}tD
+c_{3,0}D^3.
\]

Let \(F=F(t)\) be an absorbing boundary for the Heat problem. We use the
ordered boundary jet
\[
V_3^H
=
\begin{pmatrix}
v^{(4)}\\
v^{(3)}\\
v^{(2)}\\
v^{(1)}
\end{pmatrix}.
\]

The restriction of
\[
N_3^Hv=0
\]
to \(x=F(t)\) gives
\[
c_{3,0}v^{(3)}
+
c_{0,1}t\,v^{(1)}
=
0.
\tag{H$_1$}
\]

The absorbing boundary condition and the Heat equation give
\[
v^{(2)}+2F'v^{(1)}=0.
\tag{H$_2$}
\]

Differentiating this relation with respect to \(t\), while accounting for
the motion of the boundary, yields
\[
v^{(4)}
+
4F'v^{(3)}
+
4(F')^2v^{(2)}
+
4F''v^{(1)}
=
0.
\tag{H$_3$}
\]

Finally, differentiating the transported equation once with respect to
\(x\) and restricting to the boundary gives
\[
c_{3,0}v^{(4)}
+
c_{0,1}t\,v^{(2)}
+
(c_{0,0}+c_{0,1}F)v^{(1)}
=
0.
\tag{H$_4$}
\]

Thus \((\mathrm{H}_1)\)--\((\mathrm{H}_4)\) form the homogeneous system
\[
M_3^H(F;t)V_3^H=0,
\]
where
\[
M_3^H(F;t)
=
\begin{pmatrix}
0 & c_{3,0} & 0 & c_{0,1}t
\\
0 & 0 & 1 & 2F'
\\
1 & 4F' & 4(F')^2 & 4F''
\\
c_{3,0} & 0 & c_{0,1}t & c_{0,0}+c_{0,1}F
\end{pmatrix}.
\]

A nontrivial boundary jet requires
\[
\det M_3^H(F;t)=0.
\]
Direct expansion gives
\[
\det M_3^H
=
c_{3,0}
\left[
c_{0,0}
+c_{0,1}F
+2c_{0,1}tF'
+8c_{3,0}(F')^3
-4c_{3,0}F''
\right].
\]
Since \(c_{3,0}\neq0\), the Heat boundary satisfies
\[
c_{0,0}
+c_{0,1}F
+2c_{0,1}tF'
+8c_{3,0}(F')^3
-4c_{3,0}F''
=
0.
\tag{H}
\]

\subsection{Recovery of the Heat compatibility equation}

Proposition~\ref{prop:HierarchicalRecovery} guarantees that the pure Heat
spatial-jet equations are contained in the Heat--Airy hierarchy on the
slice \(s=0\). In the present example, something stronger occurs: the
additional Airy boundary identity permits the pure Heat compatibility
condition to be recovered sequentially from the first two levels of the
Heat--Airy hierarchy.

Let
\[
F(t)=f(t,0).
\]
The first Heat--Airy compatibility equation gives
\[
f_s(t,0)
=
-\frac{c_{0,1}}{3c_{3,0}}\,t.
\]
On the other hand, restricting \((\mathrm{HA}_2)\) to \(s=0\) gives
\[
\begin{aligned}
0={}&
c_{0,0}
+c_{0,1}F
-2c_{0,1}tF'
-4c_{3,0}F''
+8c_{3,0}(F')^3
\\
&
-12c_{3,0}F'f_s(t,0).
\end{aligned}
\]
Substitution of
\[
f_s(t,0)
=
-\frac{c_{0,1}}{3c_{3,0}}t
\]
gives
\[
-12c_{3,0}F'f_s(t,0)
=
4c_{0,1}tF'.
\]
Consequently,
\[
c_{0,0}
+c_{0,1}F
+2c_{0,1}tF'
+8c_{3,0}(F')^3
-4c_{3,0}F''
=
0,
\]
which is exactly the Heat compatibility equation \((\mathrm H)\).

We therefore obtain the following result.

\begin{proposition}
\label{prop:HeatRecoveryL3}
For
\[
L_3
=
c_{3,0}D^3+c_{0,0}+c_{0,1}x,
\]
the compatibility equation obtained from the pure Heat transport is
recovered sequentially from the first two Heat--Airy spatial-jet
compatibility relations on the slice \(s=0\).

More precisely, the first Heat--Airy relation determines
\[
f_s(t,0)
=
-\frac{c_{0,1}}{3c_{3,0}}t,
\]
and substitution of this identity into the second Heat--Airy relation at
\(s=0\) yields
\[
c_{0,0}
+c_{0,1}F
+2c_{0,1}tF'
+8c_{3,0}(F')^3
-4c_{3,0}F''
=
0,
\]
where
\[
F(t)=f(t,0).
\]
\end{proposition}

The distinction between
Propositions~\ref{prop:HierarchicalRecovery} and
\ref{prop:HeatRecoveryL3} is worth emphasizing. The former is a general
restriction principle: the pure Heat spatial-jet hierarchy is recovered
on \(s=0\). The latter exhibits an additional phenomenon specific to the
present operator: the Airy-flow boundary identity makes it possible to
recover the pure Heat determinant condition through a successive
elimination involving only the first two Heat--Airy levels.

\subsection{Heat boundaries as initial data for the Airy flow}

The preceding results can also be read in the opposite direction.

Suppose that \(F(t)\) is a solution of the Heat compatibility equation
\[
c_{0,0}
+c_{0,1}F
+2c_{0,1}tF'
+8c_{3,0}(F')^3
-4c_{3,0}F''
=
0.
\]
Then the first Heat--Airy compatibility equation
\[
3c_{3,0}f_s
+
2c_{0,1}s f_t
+
c_{0,1}t
=
0
\]
can be solved with initial datum
\[
f(t,0)=F(t).
\]
The characteristic formula gives
\[
f(t,s)
=
F\left(
t-\frac{c_{0,1}}{3c_{3,0}}s^2
\right)
-
\frac{c_{0,1}}{3c_{3,0}}st
+
\frac{2c_{0,1}^2}{27c_{3,0}^2}s^3.
\]

Thus the Heat and Heat--Airy constructions provide complementary
information:
\[
\boxed{
\begin{aligned}
\text{Heat compatibility}
&\quad\Longrightarrow\quad
F(t),
\\
\text{Heat--Airy evolution}
&\quad\Longrightarrow\quad
f(t,s)
\quad\text{with}\quad
f(t,0)=F(t).
\end{aligned}
}
\]

The additional Airy variable therefore does not replace the information
contained in the Heat problem. Instead, it organizes that information
within a higher-flow hierarchy and provides an explicit propagation of
the Heat boundary into an additional evolution direction. Conversely,
the restriction principle shows that the pure Heat spatial-jet equations
remain embedded in the enlarged hierarchy, while the present third-order
example demonstrates that the additional Airy identities may permit
their compatibility condition to be recovered sequentially.

This suggests a broader higher-flow principle. Adjoining commuting
higher-order evolutions may reorganize moving-boundary compatibility
conditions into a hierarchy in which different pieces of the original
one-parameter problem appear at successive jet levels. The systematic
study of this phenomenon for arbitrary polynomial differential operators
and additional commuting flows lies beyond the scope of the present
paper.

\section{Conclusion}
\label{sec:Conclusion}

We have developed a two-parameter extension of the classical heat
transport of differential operators by means of the commuting semigroup
\[
P_{t,s}
=
\exp\left(
\frac{t}{2}D^2-\frac{s}{3}D^3
\right).
\]
The corresponding transported position operator,
\[
X_{t,s}
=
x+tD-sD^2,
\]
provides the basic algebraic mechanism of the construction.

At the polynomial level, this transport gives rise to the Heat--Airy
polynomials
\[
\mathcal H_n(x,t,s)
=
P_{t,s}(x^n),
\]
together with a triangular family of normal-ordering coefficients
\[
X_{t,s}^n
=
\sum_{k=0}^{2n}
T_{n,k}(x,t,s)D^k.
\]
The first column of this array recovers the Heat--Airy polynomial
sequence, while the remaining coefficients encode the noncommutative
corrections required for the transport of polynomial differential
operators.

The second part of the paper concerns moving absorbing boundaries. For
solutions of
\[
v_t=\frac12v_{xx},
\qquad
v_s=-\frac13v_{xxx},
\]
the condition
\[
v(t,s,f(t,s))=0
\]
generates a hierarchy of identities among the spatial boundary jets of
\(v\). When these identities are combined with the boundary restrictions
of transported differential equations,
\[
\gamma_f\!\left(D^k(N_nv)\right)=0,
\qquad k\geq0,
\]
they produce differential compatibility conditions for the moving
boundary \(f\).

For second-order operators, this construction leads naturally to a
finite characteristic system. Its determinant yields a nonlinear
boundary equation whose restriction to \(s=0\) recovers the corresponding
classical Heat compatibility condition. Explicit examples show that the
first boundary equation may leave functional or parametric freedom that
is resolved by subsequent equations in the spatial jet. In particular,
higher compatibility relations determine the initial deformation of a
Heat boundary in the Airy direction.

The third-order example
\[
L_3
=
c_{3,0}D^3+c_{0,0}+c_{0,1}x
\]
makes this mechanism especially transparent. The first Heat--Airy
compatibility equation determines an explicit evolution of the boundary
away from the Heat slice,
\[
f(t,s)
=
F\left(
t-\frac{c_{0,1}}{3c_{3,0}}s^2
\right)
-
\frac{c_{0,1}}{3c_{3,0}}st
+
\frac{2c_{0,1}^2}{27c_{3,0}^2}s^3,
\]
where
\[
F(t)=f(t,0).
\]
At the same time, the first two Heat--Airy spatial-jet compatibility
relations, restricted to \(s=0\), recover precisely the nonlinear
compatibility equation obtained independently from the pure Heat
problem.

This example suggests that the additional commuting flow reorganizes,
rather than merely enlarges, the moving-boundary compatibility problem.
Information that in the one-parameter Heat theory is obtained through
the elimination of a larger spatial jet may appear, in the two-parameter
setting, across successive equations of a hierarchy.

The present framework naturally points toward further extensions. One
may consider additional commuting flows generated by higher powers of
\(D\), leading formally to operators of the form
\[
\exp\left(
\frac{t_2}{2}D^2
-\frac{t_3}{3}D^3
+\frac{t_4}{4}D^4
-\cdots
\right),
\]
with moving boundaries depending on several evolution parameters. The
associated normal-ordering problem and the corresponding hierarchy of
boundary compatibility equations suggest a broader higher-flow theory
for polynomial differential operators and absorbing boundaries. We leave
this development for future work.

\end{document}